\documentclass[preprint,12pt]{elsarticle1}

\usepackage{subfig}
\usepackage{graphicx}
\usepackage{caption,color}   
\usepackage{amssymb}
\usepackage{amsthm}
\usepackage{amsmath}
\usepackage{epic}
\usepackage{setspace}
\usepackage{float}
\usepackage{multirow}
\newtheorem{thm}{Theorem}[section]
\newtheorem{cor}[thm]{Corollary}

\newtheorem{lem}[thm]{Lemma}

\numberwithin{equation}{section}
\journal{}

\begin{document}
\begin{spacing}{1.15}
\begin{frontmatter}
\title{Extremal Zagreb indices of bicyclic hypergraphs}

\author{Hong Zhou}
\author{Changjiang Bu}\ead{buchangjiang@hrbeu.edu.cn}
\address{School of Mathematical Sciences, Harbin Engineering University, Harbin 150001, PR China}

\begin{abstract}

The Zagreb index of a hypergraph is defined as the sum of the squares of the degrees of its vertices.
A connected $k$-uniform hypergraph with $n$ vertices and $m$ edges is called bicyclic if $n=m(k-1)-1$.
In this paper, we determine the hypergraphs with the maximum and minimum Zagreb indices among all linear bicyclic uniform hypergraphs.

\end{abstract}

\end{frontmatter}

\section{Introduction}

Let $G=(V,E)$ be a simple graph with the vertex set $V(G)$ and the edge set $E(G)$.
The first Zagreb index $M_{1}$ of a graph $G$ is defined as
$$
M_{1}(G)=\sum\limits_{u\in V(G)}(d_{G}(u))^{2},
$$
where $d_{G}(u)$ is the degree of vertex $u$ in $G$. The first Zagreb index is also called Gutman index.

The first Zagreb index, first proposed by Gutman and Trinajsti{\'c} \cite{1972Graph}, is originated
from chemical studies on total $\pi$-electron energy of conjugated molecules.
The main properties of first Zagreb indices were summarized in \cite{30years,2004Graph}.
Deng \cite{Aunifiedapproach} characterized the graphs with the maximum and minimum first Zagreb indices among all bicyclic graphs.
Some results on extremal first Zagreb indices have been obtained in the literature:
see \biboptions{sort&compress}\cite{2004Graph,treedegree,treesdomination,treesdistancedomination} for trees, \cite{Uncyclicgraphswiththefirstthree,orderUncyclicgraphs} for unicyclic graphs,
\cite{Zagrebindicesquasiunicyclic}  for $k$-generalized quasi unicyclic graphs,
\cite{Zagrebindices} for triangle-free graphs, and \biboptions{sort&compress}\cite{Zagrebindicescut,Zagrebindicescutedges,Zagrebindicesclique,Zagrebindicesconnectivity,Zagrebindicespendent} for graphs with given parameters.



Hypergraphs find application in chemistry when modeling molecules or
chemical reactions involving multiple atoms bonding simultaneously.
Hypergraphs offer a more accurate depiction of certain
chemical scenarios, such as transition states in reactions, which involve
multiple atoms simultaneously changing their bonding configurations. The model for an organometallic
compound, where the hyperedges with two vertices represent covalent
bond and with more than two vertices represent delocalised polycentric
bonds, was studied in \cite{Applicationfhypergraph}.

For a hypergraph $\mathcal{H}$, $$\sum\limits_{u\in V(\mathcal{H})}(d_{\mathcal{H}}(u))^{2}$$ is called the Zagreb index of $\mathcal{H}$, denoted by $M(\mathcal{H})$, where $d_{\mathcal{H}}(u)$ is the degree of vertex $u$ in $\mathcal{H}$ \cite{Kau2020Energies}.
The hypergraphs with the maximum and minimum Zagreb indices were determined both for uniform hypertrees and for linear unicyclic uniform hypergraphs \cite{2309.16925}.
The bounds on the Zagreb indices of hypergraphs, weak bipartite hypergraphs, hypertrees, $k$-uniform
hypergraphs, $k$-uniform weak bipartite hypergraphs, and $k$-uniform hypertrees were given \cite{zagrebhy}.

In this paper, we characterize the hypergraphs with the maximum and minimum Zagreb indices among all linear bicyclic uniform hypergraphs.

\section{Preliminaries}

A hypergraph $\mathcal{H}$ is called \textit{$k$-uniform} if every edge of $\mathcal{H}$ contains exactly $k$ vertices.
A vertex of $\mathcal{H}$ is called a \textit{cored vertex} if it has degree one.
An edge $e$ of $\mathcal{H}$ is called a \textit{pendant edge} if it contains exactly $|e|-1$ cored vertices.
A cored vertex in a pendant edge is also called a \textit{pendant vertex}.
A hypergraph $\mathcal{H}$ is called \textit{linear} if any two edges intersect into at most one vertex.
The \textit{girth} of $\mathcal{H}$ is the minimum length of the hypercycles of $\mathcal{H}$.
All linear bicyclic $k$-uniform hypergraphs with $n$ vertices and $m$ edges consist of the following two types $\mathcal{B}_{n}^{k}$ and $\mathcal{C}_{n}^{k}$.

For three integers $p, q, l$ with $q\geq p\geq3$ and $l\geq 0$, let $C_{1} (resp., C_{2})$ be a $k$-uniform
 hypercycle of length $p$ (resp., $q$) and $P = (u_{0},e_{1},u_{1},\ldots,e_{l},u_{l})$ be a $k$-uniform hyperpath
of length $l$. Let $v_{1,1}\in V(C_{1}), v_{2,1}\in V(C_{2})$ be two arbitrary vertices with degree 1, and
let $v_{1,2}\in V(C_{1}), v_{2,2}\in V(C_{2})$ be two arbitrary vertices with degree 2. For $n'= (p+q+l)(k-1)-1$, let $B_{1,n',p,l,q}^{k}$
be the $n'$-vertex $k$-uniform bicyclic hypergraph obtained by
identifying $v_{1,2}$ with $u_{0}$, and identifying $v_{2,2}$ with $u_{l}$, let $B_{2,n',p,l,q}^{k}$
be the $n'$-vertex $k$-uniform
bicyclic hypergraph obtained by identifying $v_{1,2}$ with $u_{0}$, and identifying $v_{2,1}$ with $u_{l}$, and
let $B_{3,n',p,l,q}^{k}$
be the $n'$-vertex $k$-uniform bicyclic hypergraph obtained by identifying $v_{1,1}$ with $u_{0}$, and identifying $v_{2,1}$ with $u_{l}$.

For $n\geq n'$ and $i\in \{1, 2, 3\}$, let $\mathcal{B}_{i,n,p,l,q}^{k}$ be the set of $n$-vertex $m$-edge $k$-uniform bicyclic hypergraphs
each of which contains $B_{i,n',p,l,q}^{k}$ as a sub-hypergraph, where $m=\frac{n+1}{k-1}$. 
Let $\mathcal{B}_{n}^{k}=\bigcup_{i=1}^{3}\{\mathcal{B}_{i,n,p,l,q}^{k}~|~q\geq p\geq3,l\geq 0\}$.
Moreover, let $B_{1,n,p,0,q}^{k}(m-p-q)$ denote the $k$-uniform bicyclic hypergraph
obtained from $B_{1,n',p,0,q}^{k}$ by attaching $m-p-q$ pendant edges at the unique vertex with
degree 4, where $m=\frac{n+1}{k-1}$.

Let $P_{p}=(u_{1},e_{1},u_{2},\ldots,e_{p},u_{p+1}), P_{q}=(v_{1}, f_{1}, v_{2},\ldots, f_{q}, v_{q+1})$ and $P_{l}=(w_{1},g_{1},\\w_{2},\ldots,g_{l},w_{l+1})$ be three $k$-uniform hyperpaths, and suppose $(p+q+l)(k-1)-1=n'$. For three integers $p, q, l$ with $p=1,1<q\leq l$ and $1<p\leq q\leq l$, let $C_{1,n',p,q,l}^{k}$ be the $n'$-vertex $k$-uniform bicyclic hypergraph obtained from $P_{p}, P_{q}$ and $P_{l}$
by identifying three vertices $u_{1}, v_{1}, w_{1},$ and identifying three vertices $u_{p+1}, v_{q+1}, w_{l+1}$.
For $q>1, 1\leq p\leq q-1\leq l$ and $q=1, 1<p\leq l$, let $C_{2,n',p,q,l}^{k}$ be the $n'$-vertex $k$-uniform bicyclic hypergraph obtained from $P_{p}, P_{q}$ and
$P_{l}$ by identifying three vertices $u_{1}, v_{1}, w_{1}$, identifying $u_{p+1}$ with $v_{q+1}$, and identifying
$w_{l+1}$ with $v$, respectively, where $v\in f_{q}\setminus \{v_{q}, v_{q+1}\}$.
For $q>2, 1\leq p\leq q-2\leq l$ and $q=2,1\leq p\leq l$ and $q=1, k>3,1<p\leq l$, let $C_{3,n',p,q,l}^{k}$ be the $n'$-vertex $k$-uniform bicyclic hypergraph obtained from $P_{p}, P_{q}$ and
$P_{l}$ by identifying $u_{1}$ with $v_{1}$, identifying $u_{p+1}$ with $v_{q+1}$, identifying
$w_{1}$ with $v'$, and identifying $w_{l+1}$ with $v''$, respectively, where $v'\in f_{1}\setminus \{v_{1}, v_{2}\}$ and $v''\in f_{q}\setminus \{v_{q}, v_{q+1}\}$.
(in the special case $q=1$, we choose $v'\neq v''$).

For $n\geq n'$ and $i\in \{1, 2, 3\}$, let $\mathcal{C}_{i,n,p,q,l}^{k}$ be the set of $n$-vertex $m$-edge $k$-uniform bicyclic hypergraphs
each of which contains $C_{i,n',p,q,l}^{k}$ as a sub-hypergraph, where $m=\frac{n+1}{k-1}$.
Let $\mathcal{C}_{n}^{k}=\{\mathcal{C}_{1,n,p,q,l}^{k} ~|~ p=1,1<q\leq l \text{~or~} 1<p\leq q\leq l\} \bigcup \{\mathcal{C}_{2,n,p,q,l}^{k}~|~ q=1,1<p\leq l \text{~or~} q>1, 1\leq p\leq q-1\leq l\} \bigcup \{\mathcal{C}_{3,n,p,q,l}^{k}~|~ q>2, 1\leq p\leq q-2\leq l\text{~or~} q=2,1\leq p\leq l  \text {~or~}q=1, k>3,1<p\leq l\}$.
Moreover, for $i\in \{1, 2\}$, let $C_{i,n,p,q,l}^{k}(m-p-q-l)$  denote the $k$-uniform bicyclic hypergraph
obtained from $C_{i,n',p,q,l}^{k}$ by attaching $m-p-q-l$ pendant edges at the vertex with
degree 3, where $m=\frac{n+1}{k-1}$.


The following Lemma gives an operation of moving edges which changes the Zagreb indices.


\begin{lem}\label{sp277}
For $k\geq 3$, let $\mathcal{H}$ be a linear $k$-uniform hypergraph with $u,v\in V(\mathcal{H})$ and $d_{\mathcal{H}}(u)\geq 2$. Let $e_{1},\ldots,e_{t}$ be the edges incident with $u$, $v\notin e_{i}$ for each $i\in[t]$  and $d_{\mathcal{H}}(v)>d_{\mathcal{H}}(u)-t$.  Write $e'_{i}=(e_{i}\setminus \{u\})\bigcup\{v\}$ for each $i\in[t]$. Let $\mathcal{H}^{'}$ be the hypergraph with $V(\mathcal{H}^{'})=V(\mathcal{H})$ and $E(\mathcal{H}^{'})=(E(\mathcal{H})\setminus\{e_{1},\ldots,e_{t}\})\bigcup\{e'_{1},\ldots,e'_{t}\}$.
Then $\mathcal{H}^{'}$ is obtained from $\mathcal{H}$ by moving $t$ edges from $u$ to $v$ and $M(\mathcal{H}')>M(\mathcal{H})$.
\end{lem}
\begin{proof}
By the definition of the Zagreb index, we have
\begin{align*}
M(\mathcal{H}')-M(\mathcal{H})&=d^{2}_{\mathcal{H}'}(v)+d^{2}_{\mathcal{H}'}(u)-d^{2}_{\mathcal{H}}(v)-d^{2}_{\mathcal{H}}(u)\\
&=(d_{\mathcal{H}}(v)+t)^{2}+(d_{\mathcal{H}}(u)-t)^{2}-d_{\mathcal{H}}^{2}(v)-d_{\mathcal{H}}^{2}(u)\\
&=2t(t+d_{\mathcal{H}}(v)-d_{\mathcal{H}}(u))>0.
\end{align*}
\end{proof}


\section{Main results}

In this section,
we give the hypergraphs with the maximum and minimum Zagreb indices among all linear bicyclic uniform hypergraphs.

The following Theorem gives the hypergraph with minimum Zagreb index among all linear bicyclic uniform hypergraphs.
\begin{thm}
In linear bicyclic $k$-uniform hypergraphs with $n$ vertices and $m$ edges, the hypergraph with maximum degree $2$ has minimum Zagreb index.
\end{thm}
\begin{proof}
Let $\mathcal{H}$ be a bicyclic $k$-uniform hypergraph with $n$ vertices and $m$ edges. Let $n_{t}$ be the number of vertices of $\mathcal{H}$  whose degree is equal to $t$, and $\Delta_{\mathcal{H}}$ be the maximum degree of $\mathcal{H}$. Then
\begin{equation*}\label{zsyzh1}
\sum\limits_{t=1}^{\Delta_{\mathcal{H}}}n_{t}=n,\sum\limits_{t=1}^{\Delta_{\mathcal{H}}}tn_{t}=km, \text{~and~} M(\mathcal{H})=\sum\limits_{t=1}^{\Delta_{\mathcal{H}}}t^{2}n_{t}.
\end{equation*}

By the above Equations, we have
$$
M(\mathcal{H})
=\sum\limits_{t=1}^{\Delta_{\mathcal{H}}}((t-1)(t-2)+3t-2)n_{t}
=\sum\limits_{t=1}^{\Delta_{\mathcal{H}}}(t-1)(t-2)n_{t}+3km-2n.\\
$$
Therefore, when $\Delta_{\mathcal{H}}=2$, $\mathcal{H}$ has minimum Zagreb index, and $M(\mathcal{H})=3km-2n$.

\end{proof}

The following Corollary follows immediately from the above Theorem.
\begin{cor}
In linear bicyclic $k$-uniform hypergraphs with $n$ vertices, $m$ edges and girth $g$, the hypergraph with maximum degree $2$ has minimum Zagreb index.
\end{cor}

The following Theorem
gives the hypergraph with maximum Zagreb index among all hypergraphs in $\bigcup_{i=1}^{3}\{\mathcal{B}_{i,n,g,l,q}^{k}~|~q\geq g,l\geq 0\}$, and
gives the hypergraph with maximum Zagreb index among all hypergraphs in $\mathcal{B}_{n}^{k}$.
\begin{thm}
For $k\geq3$ and $m\geq 2g$, $B_{1,n,g,0,g}^{k}(m-2g)$ is the hypergraph with maximum Zagreb index among all hypergraphs in $\bigcup_{i=1}^{3}\{\mathcal{B}_{i,n,g,l,q}^{k}~|~q\geq g,l\geq 0\}$.
And for $m\geq6$,
$B_{1,n,3,0,3}^{k}(m-6)$ is the hypergraph with maximum Zagreb index among all  hypergraphs in $\mathcal{B}_{n}^{k}$.
\end{thm}
\begin{proof}

Firstly, we consider the hypergraph  in $\mathcal{B}_{1,n,g,l,q}^{k}$. When $l>0$,
repeating the operation of moving edges of Lemma \ref{sp277}, any hypergraph in $\mathcal{B}_{1,n,g,l,q}^{k}$ can be changed into a $k$-uniform bicyclic hypergraph $\mathcal{H}_{1}$
obtained from $B_{1,n',g,l,q}^{k}$ by attaching $m-g-q-l$ pendant edges at a vertex with degree 3.
And each application of Lemma \ref{sp277} strictly increases the Zagreb index. Without loss of generality, let $d_{\mathcal{H}_{1}}(v_{2,2})=3, d_{\mathcal{H}_{1}}(v_{1,2})\geq3$.
Suppose $\mathcal{H}_{2}$ is obtained from $\mathcal{H}_{1}$ by moving 2 edges incident with $v_{2,2}$ in $E(C_{2})$
from vertex $v_{2,2}$ to vertex $v_{1,2}$. By Lemma \ref{sp277}, we have $M(\mathcal{H}_{2})>M(\mathcal{H}_{1})$.
Repeating the operation of moving edges of Lemma \ref{sp277}, $\mathcal{H}_{2}$ be changed into $B_{1,n,g,0,q}^{k}(m-g-q)$.
When $l=0$, repeating the operation of moving edges of Lemma \ref{sp277}, any hypergraph in $\mathcal{B}_{1,n,g,0,q}^{k}$ can be changed into $B_{1,n,g,0,q}^{k}(m-g-q)$.

When $q=g+s$ and $s>0$. The hypergraph
$B_{1,n,g,0,q-1}^{k}(m-g-q+1)$ can be obtained from $B_{1,n,g,0,q}^{k}(m-g-q)$ by moving 1 edges in $E(C_{2})$
from a vertex with degree $2$ adjacent to $v_{1,2}$ to $v_{1,2}$. By Lemma \ref{sp277}, we have $M(B_{1,n,g,0,q-1}^{k}(m-g-q+1))>M(B_{1,n,g,0,q}^{k}(m-g-q))$.
Similarly, we have $M(B_{1,n,g,0,q}^{k}(m-g-q))<\cdots<M(B_{1,n,g,0,q-s+1}^{k}(m-g-q+s-1))<B_{1,n,g,0,g}^{k}(m-2g)$.
Therefore, $B_{1,n,g,0,g}^{k}(m-2g)$ is the hypergraph with maximum Zagreb index in $\{\mathcal{B}_{1,n,g,l,q}^{k}~|~l\geq0,q\geq g\}$.

Secondly, we consider the hypergraph in $\mathcal{B}_{2,n,g,l,q}^{k}$.
Repeating the operation of moving edges of Lemma \ref{sp277}, any hypergraph in $\mathcal{B}_{2,n,g,l,q}^{k}$ can be changed into a $k$-uniform bicyclic hypergraph $\mathcal{H}_{3}$
obtained from $B_{2,n',g,l,q}^{k}$ by attaching $m-g-q-l$ pendant edges at the vertex with degree 3.

When $l>0$, $\mathcal{H}_{4}$ is obtained from $\mathcal{H}_{3}$ by moving $e_{l}$
from vertex $v_{2,1}$ to vertex $v_{2,2}$. By Lemma \ref{sp277}, we have $M(\mathcal{H}_{4})>M(\mathcal{H}_{3})$. Obviously, $\mathcal{H}_{4}\in \mathcal{B}_{1,n,g,l,q}^{k}.$

When $l=0$, $\mathcal{H}_{4}$ is obtained from $\mathcal{H}_{3}$ by moving $m-g-q+2$ edges in $E(\mathcal{H}_{3})-E(C_{2})$
from vertex $v_{2,1}$ to vertex $v_{2,2}$. By Lemma \ref{sp277}, we have $M(\mathcal{H}_{4})>M(\mathcal{H}_{3})$. Obviously, $\mathcal{H}_{4}\in \mathcal{B}_{1,n,g,0,q}^{k}.$
Therefore, $B_{1,n,g,0,g}^{k}(m-2g)$ is the hypergraph with maximum Zagreb index in $\bigcup_{i=1}^{2}\{\mathcal{B}_{i,n,g,l,q}^{k}~|~l\geq0,q\geq g\}$.

Thirdly, we consider the hypergraph in $\mathcal{B}_{3,n,g,l,q}^{k}$. Repeating the operation of moving edges of Lemma \ref{sp277}, any hypergraph in $\mathcal{B}_{3,n,g,l,q}^{k}$ can be changed into a $k$-uniform bicyclic hypergraph $\mathcal{H}_{5}$ obtained from $B_{3,n',g,l,q}^{k}$ by attaching $m-g-q-l$ pendant edges at the vertex with degree 2.

When $l>0$. If $\mathcal{H}_{5}$ is obtained from $B_{3,n',g,l,q}^{k}$ by attaching $m-g-q-l$ pendant edges at the vertex $v_{1,1}$. Let $\mathcal{H}_{6}$
be obtained from $\mathcal{H}_{5}$ by moving $m-g-q-l+1$ edges in $E(\mathcal{H}_{5})-E(C_{1})$
from vertex $v_{1,1}$ to vertex $v_{1,2}$. By Lemma \ref{sp277}, we have $M(\mathcal{H}_{6})>M(\mathcal{H}_{5})$. Obviously, $\mathcal{H}_{6}\in \mathcal{B}_{2,n,g,l,q}^{k}.$
If $\mathcal{H}_{5}$ is obtained from $B_{3,n',g,l,q}^{k}$ by attaching $m-g-q-l$ pendant edges at the vertex except $v_{1,1}$ with degree 2. Let $\mathcal{H}_{6}$
be obtained from $\mathcal{H}_{5}$ by moving $e_{1}$
from vertex $v_{1,1}$ to vertex $v_{1,2}$. By Lemma \ref{sp277}, we have $M(\mathcal{H}_{6})>M(\mathcal{H}_{5})$. Obviously, $\mathcal{H}_{6}\in \mathcal{B}_{2,n,g,l,q}^{k}.$

%

When $l=0$.
If $\mathcal{H}_{5}$ is obtained from $B_{3,n',g,0,q}^{k}$ by attaching $m-g-q$ pendant edges at the vertex $v_{1,1}$. Let $\mathcal{H}_{7}$ be obtained from $\mathcal{H}_{5}$ by moving $m-g-q+1$ edges in $E(\mathcal{H}_{5})-E(C_{1})$
from vertex $v_{1,1}$ to vertex $v_{1,2}$. By Lemma \ref{sp277}, we have $M(\mathcal{H}_{7})>M(\mathcal{H}_{5})$. Obviously, $\mathcal{H}_{7}\in \mathcal{B}_{2,n,g,0,q}^{k}.$
If $\mathcal{H}_{5}$ is obtained from $B_{3,n',g,0,q}^{k}$ by attaching $m-g-q$ pendant edges at the vertex except $v_{1,1}$ with degree 2. Let $\mathcal{H}_{7}$ be obtained from $\mathcal{H}_{5}$ by moving $1$ edge in $E(C_{2})$
from vertex $v_{1,1}$ to vertex $v_{1,2}$. By Lemma \ref{sp277}, we have $M(\mathcal{H}_{7})>M(\mathcal{H}_{5})$. Obviously, $\mathcal{H}_{7}\in \mathcal{B}_{2,n,g,0,q}^{k}.$

Therefore, $B_{1,n,g,0,g}^{k}(m-2g)$ is the hypergraph with maximum Zagreb index in
$\bigcup_{i=1}^{3}\{\mathcal{B}_{i,n,g,l,q}^{k}~|~l\geq0,q\geq g\}$.

For $3\leq g\leq\frac{m}{2}$, we have
\begin{align*}
M(B_{1,n,g,0,g}^{k}(m-2g))&=2g(k-2)+(m-2g)(k-1)+8(g-1)+(m-2g+4)^{2}\\
&=-10g+mk+7m+8+m^{2}+4g^{2}-4mg.
\end{align*}
Let $f(x)=-10x+mk+7m+8+m^{2}+4x^{2}-4mx, x\in[3,\frac{m}{2}].$ Since $\frac{df(x)}{dx}=-10+8x-4m<0$, $f(x)$ is a strictly monotone decreasing function. So $M(B_{1,n,g,0,g}^{k}(m-2g)\leq M(B_{1,n,3,0,3}^{k}(m-6)$ for $3\leq g\leq\frac{m}{2}$ with the equation if and only if $g=3$. Hence, $B_{1,n,3,0,3}^{k}(m-6)$ is the hypergraph with maximum Zagreb index among all hypergraphs in $\mathcal{B}_{n}^{k}=\bigcup_{i=1}^{3}\{\mathcal{B}_{i,n,p,l,q}^{k}~|~q\geq p\geq3,l\geq 0\}$.

\end{proof}

The following Theorem gives the hypergraph with maximum Zagreb index among  all hypergraphs with girth $g$ in $\mathcal{C}_{n}^{k}$, and gives the hypergraph with maximum Zagreb index among all hypergraphs in $\mathcal{C}_{n}^{k}$.
\begin{thm}\label{q1}
For $k\geq3$ and $m\geq \frac{3g}{2}$, when $g$ is even, $C_{1,n,\frac{g}{2},\frac{g}{2},\frac{g}{2}}^{k}(m-\frac{3g}{2})$ is the hypergraph with maximum Zagreb index among all hypergraphs with girth $g$ in $\mathcal{C}_{n}^{k}$.
When $g$ is odd, $C_{2,n,\lfloor\frac{g}{2}\rfloor,\lceil\frac{g}{2}\rceil,\lfloor\frac{g}{2}\rfloor}^{k}(m-g-\lfloor\frac{g}{2}\rfloor)$ is the hypergraph with maximum Zagreb index among all  hypergraphs with girth $g$ in $\mathcal{C}_{n}^{k}$.

For $m\geq 6$, $C_{2,n,1,2,1}^{k}(m-4)$ is the hypergraph with maximum Zagreb index among all hypergraphs in $\mathcal{C}_{n}^{k}$.
\end{thm}
\begin{proof}
Firstly, we consider the hypergraph in $\mathcal{C}_{1,n,p,g-p,l}^{k}$.
Repeating the operation of moving edges of Lemma \ref{sp277}, any hypergraph in $\mathcal{C}_{1,n,p,g-p,l}^{k}$ can be changed into a $k$-uniform bicyclic hypergraph $C_{1,n,p,g-p,l}^{k}(m-g-l)$
obtained from $C_{1,n',p,g-p,l}^{k}$ by attaching $m-g-l$ pendant edges at a vertex with degree 3.
And each application of Lemma \ref{sp277} strictly increases the Zagreb index.

If $g$ is even. When $g-p<l$, without loss of generality, let $d_{C_{1,n,p,g-p,l}^{k}(m-g-l)}(u_{1})\geq3$.
The hypergraph
$C_{1,n,p,g-p,l-1}^{k}(m-g-l+1)$ can
be obtained from $C_{1,n,p,g-p,l}^{k}(m-g-l)$ by moving $g_{2}$
from vertex $w_{2}$ to vertex $u_{1}$. By Lemma \ref{sp277}, we have $M(C_{1,n,p,g-p,l-1}^{k}(m-g-l+1))>M(C_{1,n,p,g-p,l}^{k}(m-g-l))$.
Similarly, we get $M(C_{1,n,p,g-p,l}^{k}(m-g-l))<M(C_{1,n,p,g-p,l-1}^{k}(m-g-l+1))<\cdots<M(C_{1,n,p,g-p,g-p}^{k}(m-2g+p)).$

When $p\neq \frac{g}{2}$, we have
\begin{align*}
  &M(C_{1,n,p+1,g-p-1,g-p-1}^{k}(m-2g+p+1))-M(C_{1,n,p,g-p,g-p}^{k}(m-2g+p))\\
  &=(m-2g+p+4)^{2}+1^{2}-(m-2g+p+3)^{2}-2^{2}\\
   &=2(m-2g+p)+4>0.
\end{align*}
Similarly, we get $M(C_{1,n,p,g-p,g-p}^{k}(m-2g+p))<M(C_{1,n,p+1,g-p-1,g-p-1}^{k}(m-2g+p+1))<\cdots<M(C_{1,n,\frac{g}{2},\frac{g}{2},\frac{g}{2}}^{k}(m-\frac{3g}{2}))$. Therefore, when $g$ is even, $C_{1,n,\frac{g}{2},\frac{g}{2},\frac{g}{2}}^{k}(m-\frac{3g}{2})$ has maximum Zagreb index among all hypergraphs in $\{\mathcal{C}_{1,n,p,g-p,l}^{k}~|~ p=1,1<g-p\leq l \text{~or~} 1<p\leq g-p\leq l\}$.

If $g$ is odd, similarly, we get $C_{1,n,\lfloor\frac{g}{2}\rfloor,\lceil\frac{g}{2}\rceil,\lceil\frac{g}{2}\rceil}^{k}(m-g-\lceil\frac{g}{2}\rceil)$ has maximum Zagreb index among all hypergraphs  in $\{\mathcal{C}_{1,n,p,g-p,l}^{k}~|~ p=1,1<g-p\leq l \text{~or~} 1<p\leq g-p\leq l\}$.

Secondly, for $p+q=g$, we consider the hypergraph in $\mathcal{C}_{2,n,p,q,l}^{k}$. Repeating the operation of moving edges of Lemma \ref{sp277}, any hypergraph in $\mathcal{C}_{2,n,p,q,l}^{k}$ can be changed into a $k$-uniform bicyclic hypergraph $\mathcal{H}_{1}$
obtained from $C_{2,n',p,q,l}^{k}$ by attaching $m-p-q-l$ pendant edges at the vertex with degree 3.

If $q$ of $\mathcal{H}_{1}$ is equal to $1$, then the girth is $p+1$.
$\mathcal{H}_{2}$ is obtained from $\mathcal{H}_{1}$ by moving $g_{l}$
from vertex $v$ to vertex $v_{2}$. Obviously, $\mathcal{H}_{2}\in \mathcal{C}_{1,n,1,p,l}^{k}$ and $g(\mathcal{H}_{2})=p+1$. By Lemma \ref{sp277}, we have $ M(\mathcal{H}_{2})>M(\mathcal{H}_{1})$.

If $q$ of $\mathcal{H}_{1}$ is greater than 1, then $1\leq p\leq q-1\leq l$.

When $1\leq p\leq q-1=l$.
If $1\leq p<q-1=l$, $\mathcal{H}_{2}'$ is obtained from $\mathcal{H}_{1}$ by moving $g_{l}$
from vertex $v$ to vertex $v_{q}$. Obviously, $\mathcal{H}_{2}'\in \mathcal{C}_{1,n,p+1,q-1,l}^{k}$ and $g(\mathcal{H}_{2}')=p+q$. By Lemma \ref{sp277}, we have $ M(\mathcal{H}_{2}')>M(\mathcal{H}_{1})$.

If $1\leq p=q-1=l$, then $g(\mathcal{H}_{1})=2l+1$. Obviously, the girth of $\mathcal{H}_{1}$ is odd. And $\mathcal{H}_{1}=C_{2,n,l,l+1,l}^{k}(m-3l-1)=C_{2,n,\lfloor\frac{g}{2}\rfloor,\lceil\frac{g}{2}\rceil,\lfloor\frac{g}{2}\rfloor}^{k}(m-g-\lfloor\frac{g}{2}\rfloor)$. We have
\begin{align*}
&M(C_{2,n,\lfloor\frac{g}{2}\rfloor,\lceil\frac{g}{2}\rceil,\lfloor\frac{g}{2}\rfloor}^{k}(m-g-\lfloor\frac{g}{2}\rfloor))\\
&=(g+\frac{g-1}{2}-1)(k-2)+(k-3)+(m-g-\frac{g-1}{2})(k-1)+4(g+\frac{g-1}{2}-1)\\
&+(3+m-g-\frac{g-1}{2})^{2}\\
&=-6g+\frac{23}{4}+mk+6m+m^{2}+\frac{9g^{2}}{4}-3mg.
\end{align*}

When the girth is odd, $C_{1,n,\lfloor\frac{g}{2}\rfloor,\lceil\frac{g}{2}\rceil,\lceil\frac{g}{2}\rceil}^{k}(m-g-\lceil\frac{g}{2}\rceil)$ has maximum Zagreb index among all  hypergraphs  in $\{\mathcal{C}_{1,n,p,g-p,l}^{k}~|~ p=1,1<g-p\leq l \text{~or~} 1<p\leq g-p\leq l\}$. We have
\begin{align*}
&M(C_{1,n,\lfloor\frac{g}{2}\rfloor,\lceil\frac{g}{2}\rceil,\lceil\frac{g}{2}\rceil}^{k}(m-g-\lceil\frac{g}{2}\rceil))\\
&=(g+\frac{g+1}{2})(k-2)+(m-g-\frac{g+1}{2})(k-1)+4(g+\frac{g+1}{2}-3)+9+(3+m-g-\frac{g+1}{2})^{2}\\
&=-3g+\frac{19}{4}+mk+4m+m^{2}+\frac{9}{4}g^{2}-3mg.
\end{align*}
Hence, $M(C_{1,n,\lfloor\frac{g}{2}\rfloor,\lceil\frac{g}{2}\rceil,\lceil\frac{g}{2}\rceil}^{k}(m-g-\lceil\frac{g}{2}\rceil))-M(C_{2,n,\lfloor\frac{g}{2}\rfloor,\lceil\frac{g}{2}\rceil,\lfloor\frac{g}{2}\rfloor}^{k}(m-g-\lfloor\frac{g}{2}\rfloor))=3g-1-2m<0$.

When $1\leq p\leq q-1<l$,
$\mathcal{H}_{2}''$ is obtained from $\mathcal{H}_{1}$ by moving $g_{l}$
from vertex $v$ to vertex $v_{q+1}$. Obviously, $\mathcal{H}_{2}''\in \mathcal{C}_{1,n,p,g-p,l}^{k}$ and $g(\mathcal{H}_{2}'')=p+q$. By Lemma \ref{sp277}, we have $ M(\mathcal{H}_{2}'')>M(\mathcal{H}_{1})$.


Therefore,
when $g$ is even, $C_{1,n,\frac{g}{2},\frac{g}{2},\frac{g}{2}}^{k}(m-\frac{3g}{2})$ is the hypergraph with maximum Zagreb index among all hypergraphs with girth $g$ in $\{\mathcal{C}_{1,n,p,q,l}^{k} ~|~ p=1,1<q\leq l \text{~or~} 1<p\leq q\leq l\} \bigcup \{\mathcal{C}_{2,n,p,q,l}^{k}~|~ q=1,1<p\leq l \text{~or~} q>1, 1\leq p\leq q-1\leq l\}$. When $g$ is odd, $C_{2,n,\lfloor\frac{g}{2}\rfloor,\lceil\frac{g}{2}\rceil,\lfloor\frac{g}{2}\rfloor}^{k}(m-g-\lfloor\frac{g}{2}\rfloor)$ is the hypergraph with maximum Zagreb index among all hypergraphs with girth $g$ in $\{\mathcal{C}_{1,n,p,q,l}^{k} ~|~ p=1,1<q\leq l \text{~or~} 1<p\leq q\leq l\} \bigcup \{\mathcal{C}_{2,n,p,q,l}^{k}~|~ q=1,1<p\leq l \text{~or~} q>1, 1\leq p\leq q-1\leq l\} $.

Thirdly, for $p+q=g$, we consider the hypergraph in $\mathcal{C}_{3,n,p,q,l}^{k}$. Repeating the operation of moving edges of Lemma \ref{sp277}, any hypergraph in $\mathcal{C}_{3,n,p,q,l}^{k}$ can be changed into a $k$-uniform bicyclic hypergraph $\mathcal{H}_{3}$
obtained from $C_{3,n',p,q,l}^{k}$ by attaching $m-p-q-l$ pendant edges at a vertex with degree 2.

If $q$ of $\mathcal{H}_{3}$ is equal to 1, then $g(\mathcal{H}_{3})=p+1$.
If $\mathcal{H}_{3}$ is obtained from $C_{3,n',p,1,l}^{k}$ by attaching $m-p-1-l$ pendant edges at the vertex $v'$ (or $v''$). Let $\mathcal{H}_{4}$ be obtained from $\mathcal{H}_{3}$ by moving $m-p-1-l$ pendant edges from vertex $v'$ to vertex $v_{1}$ (or from vertex $v''$ to vertex $v_{2}$), moving $g_{1}$
from vertex $v'$ to vertex $v_{1}$ and moving $g_{l}$
from vertex $v''$ to vertex $v_{2}$. Obviously, $\mathcal{H}_{4}\in \mathcal{C}_{1,n,1,p,l}^{k}$ and $g(\mathcal{H}_{4})=p+1$.
By Lemma \ref{sp277}, we have $M(\mathcal{H}_{4})>M(\mathcal{H}_{3})$.
If $\mathcal{H}_{3}$ is obtained from $C_{3,n',p,1,l}^{k}$ by attaching $m-p-1-l$ pendant edges at the vertex except $v',v''$ with degree 2.  Let $\mathcal{H}_{4}$ be obtained from $\mathcal{H}_{3}$ by moving $g_{1}$ from vertex $v'$ to vertex $v_{1}$ and moving $g_{l}$
from vertex $v''$ to vertex $v_{2}$. Obviously, $\mathcal{H}_{4}\in \mathcal{C}_{1,n,1,p,l}^{k}$ and $g(\mathcal{H}_{4})=p+1$. By Lemma \ref{sp277}, we have $M(\mathcal{H}_{4})>M(\mathcal{H}_{3})$.

If $q=2, p=l=1$ of $\mathcal{H}_{3}$, then $g(\mathcal{H}_{3})=3$.
If $\mathcal{H}_{3}$ is obtained from $C_{3,n',1,2,1}^{k}$ by attaching $m-4$ pendant edges at the vertex $v'$. Let $\mathcal{H}_{5}$ be obtained from $\mathcal{H}_{3}$ by moving $g_{1}$ and $m-4$ pendant edges from vertex $v'$ to vertex $v_{1}$. Obviously, $\mathcal{H}_{5}\in \mathcal{C}_{2,n,1,2,1}^{k}$ and $g(\mathcal{H}_{5})=3$. By Lemma \ref{sp277}, we have $M(\mathcal{H}_{5})>M(\mathcal{H}_{3})$.
If $\mathcal{H}_{3}$ is obtained from $C_{3,n',1,2,1}^{k}$ by attaching $m-4$ pendant edges at the vertex except $v'$ with degree 2.  Let $\mathcal{H}_{5}$ be obtained from $\mathcal{H}_{3}$ by moving $g_{1}$
from vertex $v'$ to vertex $v_{1}$. Obviously, $\mathcal{H}_{5}\in \mathcal{C}_{2,n,1,2,1}^{k}$ and $g(\mathcal{H}_{5})=3$. By Lemma \ref{sp277}, we have $M(\mathcal{H}_{5})>M(\mathcal{H}_{3})$.

If $q=2, 1=p<l$ (or $q=2, 1<p\leq l$) of $\mathcal{H}_{3}$, then $g(\mathcal{H}_{3})=3$ (or $p+2$). Similar to the proof of $q=1$ of $\mathcal{H}_{3}$, $\mathcal{H}_{3}$ can be changed into $\mathcal{H}_{4}$,
$\mathcal{H}_{4}\in \mathcal{C}_{1,n,1,2,l}^{k}$ (or $\mathcal{C}_{1,n,2,p,l}^{k}$), $g(\mathcal{H}_{4})=3$ (or $p+2$) and $M(\mathcal{H}_{4})>M(\mathcal{H}_{3})$.

If $q>2$ of $\mathcal{H}_{3}$, then $1\leq p\leq q-2\leq l$.
When $q\geq 2$, $M(\mathcal{H}_{3})=-p-q+2-l+mk+3m+m^{2}+p^{2}+q^{2}+l^{2}-2mp-2mq-2ml+2pq+2pl+2ql.$

When $1\leq p\leq q-2=l$, if $1\leq p=q-2=l$, then $M(\mathcal{H}_{3})=9l+mk-m+9l^{2}-6ml+m^{2}+4$ and $g(\mathcal{H}_{3})=2l+2.$

When $g$ is even, $C_{1,n,\frac{g}{2},\frac{g}{2},\frac{g}{2}}^{k}(m-\frac{3g}{2})$ is the hypergraph with maximum Zagreb index among all hypergraphs with girth $g$ in $\{\mathcal{C}_{1,n,p,q,l}^{k} ~|~ p=1,1<q\leq l \text{~or~} 1<p\leq q\leq l\} \bigcup \{\mathcal{C}_{2,n,p,q,l}^{k}~|~ q=1,1<p\leq l \text{~or~} q>1, 1\leq p\leq q-1\leq l\} $.
We have
\begin{align*}
M(C_{1,n,\frac{g}{2},\frac{g}{2},\frac{g}{2}}^{k}(m-\frac{3g}{2}))&=\frac{3}{2}g(k-2)+(m-\frac{3}{2}g)(k-1)+12(\frac{g}{2}-1)+9+(3+m-\frac{3}{2}g)^{2}\\
&=-\frac{9}{2}g+mk+5m+6+m^{2}+\frac{9}{4}g^{2}-3mg.
\end{align*}
Hence, $M(C_{1,n,l+1,l+1,l+1}^{k}(m-3l-3))=9l+mk-m+6+m^{2}+9l^{2}-6ml.$ Therefore, $M(C_{1,n,l+1,l+1,l+1}^{k}(m-3l-3))-M(\mathcal{H}_{3})=2>0$.

If $1\leq p<q-2=l$,
supppose $\mathcal{H}_{3}$ is obtained from $C_{3,n',p,q,l}^{k}$ by attaching $m-p-q-l$ pendant edges at the vertex $v'$. Let $\mathcal{H}_{6}$ be obtained from $\mathcal{H}_{3}$ by moving $g_{1}$ and $m-p-q-l$ pendant edges from vertex $v'$ to vertex $v_{2}$. Obviously, $\mathcal{H}_{6}\in \mathcal{C}_{2,n,p+1,q-1,l}^{k}$ and $g(\mathcal{H}_{6})=p+q$. By Lemma \ref{sp277}, we have $M(\mathcal{H}_{6})>M(\mathcal{H}_{3})$.
If $\mathcal{H}_{3}$ is obtained from $C_{3,n',p,q,l}^{k}$ by attaching $m-p-q-l$ pendant edges at the vertex except $v'$ with degree 2.  Let $\mathcal{H}_{6}$ be obtained from $\mathcal{H}_{3}$ by moving $g_{1}$
from vertex $v'$ to vertex $v_{2}$. Obviously, $\mathcal{H}_{6}\in \mathcal{C}_{2,n,p+1,q-1,l}^{k}$ and $g(\mathcal{H}_{6})=p+q$. By Lemma \ref{sp277}, we have $M(\mathcal{H}_{6})>M(\mathcal{H}_{3})$.

When $1\leq p\leq q-2<l$, similar to the proof of $q=2, p=l=1$ of $\mathcal{H}_{3}$, $\mathcal{H}_{3}$ can be changed into $\mathcal{H}_{5}$,
$\mathcal{H}_{5}\in \mathcal{C}_{2,n,p,q,l}^{k}$, $g(\mathcal{H}_{5})=p+q$ and $M(\mathcal{H}_{5})>M(\mathcal{H}_{3})$.



Therefore,
when $g$ is even, $C_{1,n,\frac{g}{2},\frac{g}{2},\frac{g}{2}}^{k}(m-\frac{3g}{2})$ is the hypergraph with maximum Zagreb index among all hypergraphs with girth $g$ in $\mathcal{C}_{n}^{k}$.
When $g$ is odd, $C_{2,n,\lfloor\frac{g}{2}\rfloor,\lceil\frac{g}{2}\rceil,\lfloor\frac{g}{2}\rfloor}^{k}(m-g-\lfloor\frac{g}{2}\rfloor)$ is the hypergraph with maximum Zagreb index among all hypergraphs with girth $g$ in $\mathcal{C}_{n}^{k}$.

If $\mathcal{H} \in \{\mathcal{C}_{1,n,p,q,l}^{k} ~|~ p=1,1<q\leq l \text{~or~} 1<p\leq q\leq l\}$.
Since $g=p+q$ and $p \leq q$, $g-q \leq q$, that is $q\geq \frac{g}{2}$. Since $l\geq q$, we have $l\geq \frac{g}{2}$.
Then $p+q+l=g+l\geq \frac{3}{2}g$. Thus,
when $m\geq \frac{3}{2}g$, $\mathcal{H}$ exists.

If $\mathcal{H} \in  \{\mathcal{C}_{2,n,p,q,l}^{k}~|~ q=1,1<p\leq l \text{~or~} q>1, 1\leq p\leq q-1\leq l\}$.
Since $g=p+q$,
we have $p+q\leq p+l+1$ and $p+q\leq q+l$,
which implies $l\geq \frac{g-1}{2}$.
Then we have $p+q+l\geq g+\frac{g-1}{2}=\frac{3}{2}g-\frac{1}{2}$.
Thus, when
$m\geq \frac{3}{2}g-\frac{1}{2}$, $\mathcal{H}$ exists.
For $\frac{3g}{2}-\frac{1}{2}\leq m<\frac{3g}{2}$, we have $m=\frac{3g}{2}-\frac{1}{2}$ and $g$ is odd.
If $q=1$, $\mathcal{H}$ does not exist.
If $q>1$, $1\leq g-q\leq q-1\leq l$, then $\frac{g}{2}+\frac{1}{2}\leq q\leq l+1.$  Since $m=\frac{3g}{2}-\frac{1}{2}$, $l\geq \frac{g-1}{2}$ and $\frac{g}{2}+\frac{1}{2}\leq q\leq l+1$, $l=\frac{g-1}{2}$ and $q=\frac{g}{2}+\frac{1}{2}$.
Therefore, $\mathcal{H}=C_{2,n,\lfloor\frac{g}{2}\rfloor,\lceil\frac{g}{2}\rceil,\lfloor\frac{g}{2}\rfloor}^{k}$ and $g$ is odd. 

If $\mathcal{H} \in \{\mathcal{C}_{3,n,p,q,l}^{k}~|~ q>2, 1\leq p\leq q-2\leq l\text{~or~} q=2,1\leq p\leq l  \text {~or~}q=1, k>3,1<p\leq l\}$.
Since $g=p+q$,
we have $p+q\leq p+l+2$ and $p+q\leq q+l$,
which implies $l\geq \frac{1}{2}g-1$.
Then we have $p+q+l\geq g+\frac{1}{2}g-1=\frac{3}{2}g-1$.
Thus,
when $m\geq \frac{3}{2}g-1$, $\mathcal{H}$ exists.
If $q>2$, $1\leq g-q\leq q-2\leq l$, then $\frac{g}{2}+1\leq q \leq l+2.$
For $\frac{3g}{2}-1\leq m<\frac{3g}{2}$, when $g$ is even, $m=\frac{3g}{2}-1$.
Since $m=\frac{3g}{2}-1$, $l\geq \frac{1}{2}g-1$ and $\frac{g}{2}+1\leq q \leq l+2$, $l=\frac{1}{2}g-1$ and $q=\frac{g}{2}+1$.
Therefore, $\mathcal{H}=C_{3,n,\frac{g}{2}-1,\frac{g}{2}+1,\frac{g}{2}-1}^{k}$ and $g$ is even.
When $g$ is odd, $m=\frac{3g}{2}-\frac{1}{2}$. Since $m=\frac{3g}{2}-\frac{1}{2}$, $l\geq \frac{1}{2}g-1$ and $\frac{g}{2}+1\leq q \leq l+2$, $l=\frac{1}{2}g-\frac{1}{2}$ and $q=\frac{1}{2}g+\frac{3}{2}$. Therefore, $\mathcal{H}=C_{3,n,\frac{1}{2}g-\frac{3}{2},\frac{1}{2}g+\frac{3}{2},\frac{1}{2}g-\frac{1}{2}}^{k}$ and $g$ is odd.
If $q=1$ or $2$, for $\frac{3g}{2}-1\leq m<\frac{3g}{2}$ and $m\geq 6$, $\mathcal{H}$ does not exist.

When $g$ is even, we have
$M(C_{1,n,\frac{g}{2},\frac{g}{2},\frac{g}{2}}^{k}(m-\frac{3g}{2}))=-\frac{9}{2}g+mk+5m+6+m^{2}+\frac{9}{4}g^{2}-3mg.$
Let $f(x)=-\frac{9}{2}x+mk+5m+6+m^{2}+\frac{9}{4}x^{2}-3mx, 4\leq x\leq \frac{2m}{3}.$ Since $\frac{df(x)}{dx}=-\frac{9}{2}+\frac{9x}{2}-3m<0$, $f(x)$ is a strictly monotone decreasing function. So when $g$ is even, $M(C_{1,n,\frac{g}{2},\frac{g}{2},\frac{g}{2}}^{k}(m-\frac{3g}{2}))\leq M(C_{1,n,2,2,2}^{k}(m-6))$ for $4\leq g\leq \frac{2m}{3}$ with the equation if and only if $g=4$. Since the maximum degree of $ C_{3,n,\frac{g}{2}-1,\frac{g}{2}+1,\frac{g}{2}-1}^{k}$ is 2, $M(C_{1,n,2,2,2}^{k}(m-6))> M(C_{3,n,\frac{g}{2}-1,\frac{g}{2}+1,\frac{g}{2}-1}^{k})$ for $g=\frac{2m}{3}+\frac{2}{3}.$ 

When $g$ is odd, we have $M(C_{2,n,\lfloor\frac{g}{2}\rfloor,\lceil\frac{g}{2}\rceil,\lfloor\frac{g}{2}\rfloor}^{k}(m-g-\lfloor\frac{g}{2}\rfloor))=-6g+\frac{23}{4}+mk+6m+m^{2}+\frac{9g^{2}}{4}-3mg$.
Let $f(x)=-6x+\frac{23}{4}+mk+6m+m^{2}+\frac{9x^{2}}{4}-3mx, 3\leq x\leq \frac{2m}{3}+\frac{1}{3}.$ Since $\frac{df(x)}{dx}=-6+\frac{9x}{2}-3m<0$, $f(x)$ is a strictly monotone decreasing function. So
when $g$ is odd, $M(C_{2,n,\lfloor\frac{g}{2}\rfloor,\lceil\frac{g}{2}\rceil,\lfloor\frac{g}{2}\rfloor}^{k}(m-g-\lfloor\frac{g}{2}\rfloor))\leq M(C_{2,n,1,2,1}^{k}(m-4))$ for $3\leq g\leq \frac{2m}{3}$ with the equation if and only if $g=3$. And $M(C_{2,n,1,2,1}^{k}(m-4))>M(C_{2,n,\lfloor\frac{g}{2}\rfloor,\lceil\frac{g}{2}\rceil,\lfloor\frac{g}{2}\rfloor}^{k})$ for $g=\frac{2m}{3}+\frac{1}{3}.$
Since the maximum degree of $C_{3,n,\frac{1}{2}g-\frac{3}{2},\frac{1}{2}g+\frac{3}{2},\frac{1}{2}g-\frac{1}{2}}^{k}$ is 2, $M(C_{1,n,2,2,2}^{k}(m-6))> M(C_{3,n,\frac{1}{2}g-\frac{3}{2},\frac{1}{2}g+\frac{3}{2},\frac{1}{2}g-\frac{1}{2}}^{k})$ for $g=\frac{2m}{3}+\frac{1}{3}.$

When $m\geq6$,
since $M(C_{1,n,2,2,2}^{k}(m-6))-M(C_{2,n,1,2,1}^{k}(m-4))=16-4m<0$, $C_{2,n,1,2,1}^{k}(m-4)$ is the hypergraph with maximum Zagreb index among all hypergraphs in $\mathcal{C}_{n}^{k}$.
\end{proof}

The following Theorem gives the hypergraph with maximum Zagreb index among all linear bicyclic $k$-uniform hypergraphs with $n$ vertices, $m$ edges and girth $g$. And gives the hypergraph with maximum Zagreb index among all linear bicyclic $k$-uniform hypergraphs with $n$ vertices and $m$ edges.
\begin{thm}
For $k\geq3$ and $m\geq 2g$, when $g$ is even, $C_{1,n,\frac{g}{2},\frac{g}{2},\frac{g}{2}}^{k}(m-\frac{3g}{2})$ is the hypergraph with maximum Zagreb index among all linear bicyclic $k$-uniform hypergraphs with $n$ vertices, $m$ edges and girth $g$.
When $g$ is odd, $C_{2,n,\lfloor\frac{g}{2}\rfloor,\lceil\frac{g}{2}\rceil,\lfloor\frac{g}{2}\rfloor}^{k}(m-g-\lfloor\frac{g}{2}\rfloor)$ is the hypergraph with maximum Zagreb index among all linear bicyclic $k$-uniform hypergraphs with $n$ vertices, $m$ edges and girth $g$.

For $m\geq 6$, $C_{2,n,1,2,1}^{k}(m-4)$ is the hypergraph with maximum Zagreb index among all linear bicyclic $k$-uniform hypergraphs with $n$ vertices and $m$ edges.
\end{thm}
\begin{proof}
When $g$ is even,
$M(C_{1,n,\frac{g}{2},\frac{g}{2},\frac{g}{2}}^{k}(m-\frac{3g}{2}))-M(B_{1,n,g,0,g}^{k}(m-2g))=\frac{11}{2}g-2m-2-\frac{7}{4}g^{2}+mg$, $4\leq g\leq \frac{m}{2}$.
Let $f(x)=\frac{11}{2}x-2m-2-\frac{7}{4}x^{2}+mx$. The roots of $f(x)$ are easily obtained as $x_{1}=\frac{11+2m-2\sqrt{(m-\frac{3}{2})^{2}+14}}{7}$ and $x_{2}=\frac{11+2m+2\sqrt{(m-\frac{3}{2})^{2}+14}}{7}$. Since
$x_{1}<\frac{11+2m-2(m-\frac{3}{2})}{7}=2<4$ and $x_{2}>\frac{11+2m+2(m-\frac{3}{2})}{7}=\frac{4m+8}{7}>\frac{m}{2}$, when $4\leq g\leq \frac{m}{2}$, $M(C_{1,n,\frac{g}{2},\frac{g}{2},\frac{g}{2}}^{k}(m-\frac{3g}{2}))-M(B_{1,n,g,0,g}^{k}(m-2g))>0.$
Therefore, when $g$ is even, $C_{1,n,\frac{g}{2},\frac{g}{2},\frac{g}{2}}^{k}(m-\frac{3g}{2})$ is the hypergraph with maximum Zagreb index among all linear bicyclic $k$-uniform hypergraphs with $n$ vertices, $m$ edges and girth $g$.

When $g$ is odd,
$M(C_{2,n,\lfloor\frac{g}{2}\rfloor,\lceil\frac{g}{2}\rceil,\lfloor\frac{g}{2}\rfloor}^{k}(m-g-\lfloor\frac{g}{2}\rfloor))-M(B_{1,n,g,0,g}^{k}(m-2g))=4g-m-\frac{9}{4}-\frac{7}{4}g^{2}+mg, 3\leq g\leq \frac{m}{2}$.
Let $f(x)=4x-m-\frac{9}{4}-\frac{7}{4}x^{2}+mx$. 
The roots of $f(x)$ are easily obtained as $x_{1}=\frac{8+2m-2\sqrt{(m+\frac{1}{2})^{2}}}{7}=1$ and $x_{2}=\frac{8+2m+2\sqrt{(m+\frac{1}{2})^{2}}}{7}=\frac{4m+9}{7}$. Obviously,
$x_{1}<3, x_{2}>\frac{m}{2}.$ So, when $3\leq g\leq \frac{m}{2}$, $M(C_{2,n,\lfloor\frac{g}{2}\rfloor,\lceil\frac{g}{2}\rceil,\lfloor\frac{g}{2}\rfloor}^{k}(m-g-\lfloor\frac{g}{2}\rfloor))-M(B_{1,n,g,0,g}^{k}(m-2g))>0.$
Hence, when $g$ is odd, $C_{2,n,\lfloor\frac{g}{2}\rfloor,\lceil\frac{g}{2}\rceil,\lfloor\frac{g}{2}\rfloor}^{k}(m-g-\lfloor\frac{g}{2}\rfloor)$ is the hypergraph with maximum Zagreb index among all linear bicyclic $k$-uniform hypergraphs with $n$ vertices, $m$ edges and girth $g$.

From the proof of Theorem \ref{q1}, 
$C_{2,n,1,2,1}^{k}(m-4)$ is the hypergraph with maximum Zagreb index among all linear bicyclic $k$-uniform hypergraphs with $n$ vertices and $m$ edges.
\end{proof}

\vspace{3mm}

\noindent
\textbf{Acknowledgments}
\vspace{3mm}
\noindent

This work is supported by the National Natural Science Foundation of China (No. 12071097, No. 12371344), the Natural Science Foundation for The Excellent Youth Scholars of the Heilongjiang Province (No. YQ2022A002) and the Fundamental Research Funds for the
Central Universities.

\section*{References}
\bibliographystyle{unsrt}
\bibliography{pbib}
\end{spacing}
\end{document}